\documentclass[a4paper,12pt,oneside,openbib]{article}

\usepackage[margin=2.5cm]{geometry} % Adjust margins to center text
\usepackage[psamsfonts]{amssymb}
\usepackage{amsmath,amsfonts}
\usepackage{amsthm}
\usepackage{graphicx}
\usepackage[utf8]{inputenc} 

\usepackage{bbm}
\usepackage{subfigure}

\newtheorem{theorem}{Theorem}
\newtheorem{lemma}{Lemma}
\newtheorem{proposition}{Proposition}

\newtheorem*{corollary*}{Corollary}

\theoremstyle{definition}

\usepackage[normalem]{ulem}

\graphicspath{{images/}}

\date{\today}

\usepackage{tikz}

\begin{document}

	\title{A $q$-weighted analogue of the Trollope-Delange formula}
	\author{Aleksei~Minabutdinov\thanks{
			Department of Mathematics (RiskLab) and Department of Management, Technology, and Economics, ETH Zurich, 8092 Z\"urich, Switzerland}}
	\date{\textit{Accepted for publication in \emph{Integers}}}
	\maketitle

	\vskip 20pt
	
	\centerline{\bf Dedication}
	\noindent
	This paper is dedicated to the memory of my Teacher, Professor Anatoly Vershik. His personality, mathematical brilliance, vast knowledge, versatility, optimism, and curiosity have been deeply inspiring to me and to all who surrounded him.

	\vskip 20pt
	\centerline{\bf Abstract}
	\noindent
	Let $s(n)$ denote the number of "$1$"s in the dyadic representation of a
	positive integer $n$ and sequence $S(n) = s(1)+s(2)+\dots+s(n-1)$.
	The Trollope-Delange formula is a classic result that represents the
	sequence $S$ in terms of the Takagi function. This work extends the result by introducing a $q$-weighted analog of $s(n)$, deriving a variant of the Trollope-Delange formula for this generalization. We show that for $1/2<|q|< 1$, nondifferentiable Takagi-Landsberg functions appear, whereas for $|q|>1$, the resulting functions are differentiable almost everywhere. We further show how the result can be used to find limiting curves describing fluctuations in the ergodic theorem for the dyadic odometer.
	% \pagestyle{myheadings}
	% \markright{\smalltt INTEGERS: 24 (2024)\hfill}
	% \thispagestyle{empty}
	\vskip 30pt
	
	\section{Introduction}
	Let $s(n)$ denote the number of "$1$"s in the binary (dyadic) representation of a positive integer $n$, and define the sequence  
	\[
	S(n) = \sum_{k=1}^{n-1} s(k).
	\]
	The Trollope-Delange formula, established by Trollope in 1968 \cite{Trollope}, is a fundamental result that relates the sequence $S$ to a continuous, 1-periodic, nowhere differentiable function $\tilde{F}_1$ as follows:
	\begin{equation}\label{eq:TrollopeDelange}
		\frac{1}{n}S(n) = \frac{1}{2}\log_2 n + \frac{1}{2}\tilde{F}_1(\log_2n).
	\end{equation}
	The Takagi function $\mathcal{T}:\mathbb{R}\rightarrow [0,+\infty)$ is defined by
	\[
	\mathcal{T}(x) = \sum_{n=0}^\infty \left(\frac{1}{2}\right)^n \tau(2^n x),
	\]
	where $\tau(x) = \text{dist}(x, \mathbb{Z})$ is the distance from $x \in \mathbb{R}$ to the nearest integer. It is well known (see, e.g.,  \cite{Girgensohn2012}) that $\tilde{F}_1$ can be represented through $\mathcal{T}$ as
	\begin{equation}
		\label{eq:TD_function}
		\tilde{F}_1(t) = 1-t - 2^{1-t}\mathcal{T}(2^{-(1-t)}) \quad \text{for } 0 \leq t \leq 1.
	\end{equation}

	Delange generalized this result to number systems with base $m \geq 3$ in 1975 paper~\cite{Delange}. Various other extensions followed, involving exponential, power, and binomial sums (see \cite{Kruppel}, \cite{Girgensohn2012}, \cite{AllaartKawamura2012}).
	
	A further generalization considers weighted digital sums with real-valued weights $\gamma = (\gamma_0, \gamma_1, \dots)$. For $n$ written in binary representation as $n = \omega_0 + \omega_1 2 + \omega_2 2^2 + \dots$, define $s(n, \gamma)$ by $s(n, \gamma) = \gamma_0 \omega_0 + \gamma_1 \omega_1 + \gamma_2 \omega_2 + \dots$. For $\gamma = (1,1,1,\dots)$, the definition reduces to the standard sum-of-digits function: $s(n, (1)) = s(n)$.
	In 2005, an asymptotic version of the Trollope-Delange formula was derived for $S(n, \gamma) = \sum_{k=0}^{n-1}s(k, \gamma)$. Let $[ \,\cdot \,]$ denote the integer part.
	Authors of \cite{Larcher-et-al-2005} showed that there exists a continuous 1-periodic function $\tilde{G}_{\gamma}:\mathbb{R}\rightarrow \mathbb{R}$ such that
	\begin{equation}\label{eq:Larcher}
		S(n,\gamma) = \frac{n}{2}\sum_{i=0}^{[ \log_2n ]}\gamma_i + n\tilde{G}_\gamma(\log_2n) + o(n),
	\end{equation}
	\textit{if and only if} $\lim\limits_{i \to \infty} \gamma_i = \tilde{\gamma}$ exists. 
	Moreover, the $o(n)$-term is zero if and only if $\gamma_i \equiv \tilde{\gamma},$ and the function $\tilde{G}_{\gamma}$ on $[0,1]$ can be represented by
	\[
	\tilde{G}_{\gamma}(x) = -\frac{\tilde{\gamma}}{2}\sum_{i=-1}^\infty\frac{\tau(2^{x+i})}{2^{x+i}}.
	\]
	Unlike the exact formula \eqref{eq:TrollopeDelange}, the asymptotic expression \eqref{eq:Larcher} includes an unknown term $o(n)$. 
	If $\lim_{i \to \infty} \gamma_i = 0$, then $\tilde{G}_{\gamma} \equiv 0$, and $S(n,\gamma) - \frac{n}{2}\sum_{i=0}^{[ \log_2n ]}\gamma_i$ is only described as this remainder term $o(n)$. If $\tilde{\gamma} = \infty$, $\tilde{G}_{\gamma}$  is not defined. Expression~\eqref{eq:Larcher} was later generalized to bases $m>2$ and other moments in \cite{Hofer-et-al-2008}.  Special cases like $\gamma_n = (-1)^n, n\geq 0,$ were studied in \cite{Kruppel2008}, yielding exact expressions for $S(\,\cdot\,, (-1)^n)$. Another case, $\gamma_n = (1/2)^n$, was studied in \cite{Larcher-et-al-2005} with connection to the van der Corput sequence.%In particular, it was shown that \[\lim\limits_{n\to +\infty}\frac{1}{n} S\Big(n, \Big(\frac{1}{2^{j+1}}\Big)_{j}\Big) = 1/2.\].

	This work focuses on the  $q$-weighted case $\gamma_n:= q^{n+1}$ for a real $q$.  Let $k \in \mathbb{N}_0$ be given by its dyadic expansion $k = \omega_0 + \omega_1 2 + \omega_2 2^2 + \dots$ with digits $\omega_i \in \{0,1\}$. We denote $s(k, (q^{i+1})_i) = \sum_{i \geq 0} \omega_i q^{i+1}$ as $s_q(k)$ and define $S_q(n) = S(n, (q^n)) \equiv \sum_{k=0}^{n-1} s_q(k)$.  In Section~\ref{Sec:MainResult}, we show that for $|q| > \frac{1}{2}$, an exact generalized Trollope-Delange-type formula holds, involving periodic Takagi-Landsberg functions.
	If $1/2 < |q| < 1$, our main result, Theorem \ref{Th:mainTheorem}, refines expression \eqref{eq:Larcher}, providing an explicit formula for the \( o(n) \) term in the case of \( q \)-weights.
	If $|q|>1$, the result seems to be uncovered in the literature. Interestingly, in this case, almost everywhere differentiable functions\footnote{This contrasts, for example, with the Trollope case, where the Takagi function appears. More precisely, we show that nondifferentiable Takagi-Landsberg functions arise for $1/2<|q| \leq 1$.} arise.  In Section \ref{Sec:complex_q}, we consider complex-valued $q$. Section \ref{Sec:LimitingCurves} further shows how our results relate to dynamical systems and a description of limiting curves for the sequence~$s_q(n)$.

	\subsection{Takagi-Landsberg Functions}
	\label{Sec:TakagiLandsberg}
	
	Let $a$ be a real parameter such that $|a|<1$. The Takagi-Landsberg function $\mathcal{T}_a:\mathbb{R}\rightarrow [0,+\infty)$ is defined by 
	\begin{equation}
		\mathcal{T}_a(x) = \sum_{n=0}^\infty a^n \tau(2^n x),
		\label{eq:TakagiLandsberg}
	\end{equation}
	where, as before, $\tau(x) = \,\mathrm{dist}(x, \mathbb{Z})$ is the distance from $x \in \mathbb{R}$ to the nearest integer. It is easy to see that the series \eqref{eq:TakagiLandsberg} converges uniformly\footnote{In several works, see for example \cite{HataYamaguti}, more general sums of the form $\sum_{n=0}^\infty c_n \tau(2^n x)$ were considered where $\sum_{n=0}^\infty |c_n|<\infty$. Such a class of functions is called the Takagi class.}, and thus defines a continuous function $\mathcal{T}_a$ for values $|a|<1$. 
	%For $2^k<n<2^{k+1}$ it is straightforward to show that 
	%\begin{equation}
	%     \mathcal{T}_a(n/2^{k+1}) = \frac{1}{a^{k+1}} \sum_{i=1}^{k+1} a^i \tau(n/2^i).
	%     \label{eq:T_a_at_dyadic_rational_points}
	% \end{equation}  
	The family of 1-periodic functions $\{\mathcal{T}_a\}_a$ proposed by \cite{Landsberg} can be considered a direct generalization of the well-known Takagi function $\mathcal{T}$, introduced by T. Takagi \cite{Takagi1903}, which is obtained when $a$ is set equal to $1/2$. For $|a|\geq\frac{1}{2}$, the functions $\mathcal{T}_a$ are  nowhere differentiable, but for $|a|<\frac{1}{2}$, they are differentiable almost everywhere\footnote{In particular, $\mathcal{T}_{1/4}(x) = x(1-x)$.}, this follows from \cite{Kono1987}, see also  the surveys \cite{AllaartKawamura2012} and \cite{Lagarias}. The functions $\mathcal{T}_a$ have been studied in several works, see \cite{Allaart11, Boros, TaborTabor2009}.
	
	\begin{figure}[t]
		\centering
		\subfigure[$a=-1/2$ (alternating sign Takagi curve)]{\includegraphics[width=0.45\linewidth]{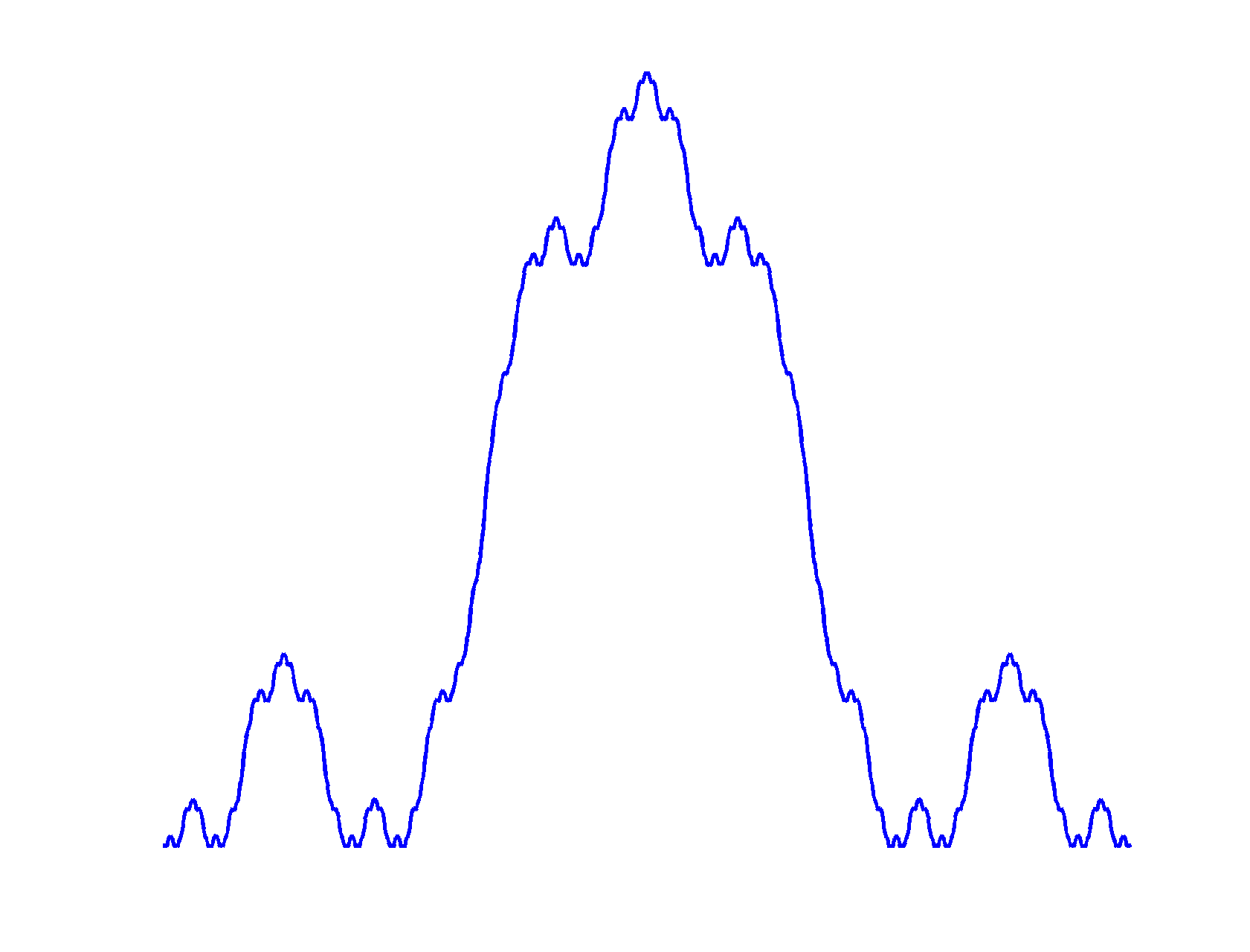}}
		\subfigure[$a=1/2$ (Takagi-Blancmange curve)]{\includegraphics[width=0.45\linewidth]{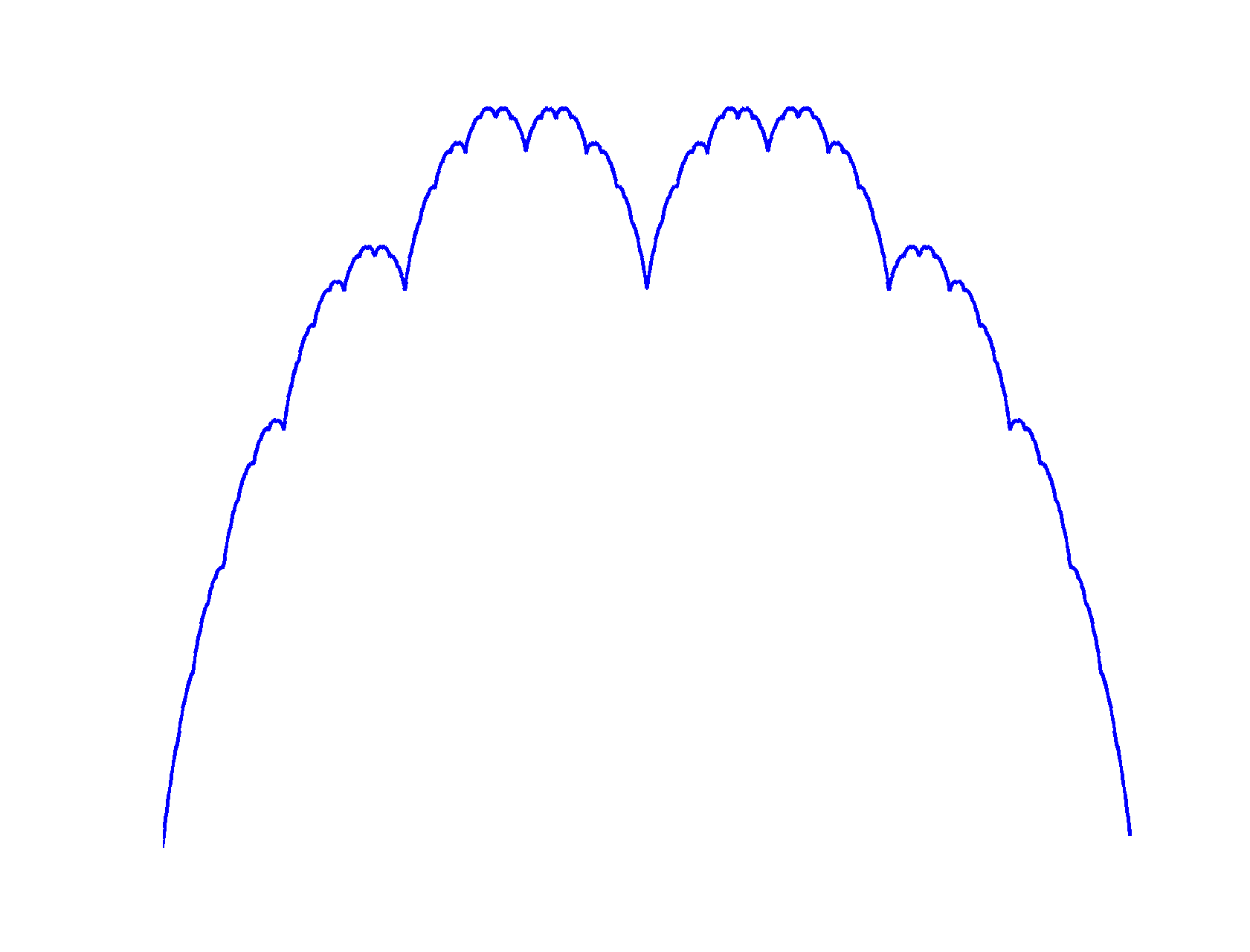}}
		\subfigure[$a=2/3$]{\includegraphics[width=0.45\linewidth]{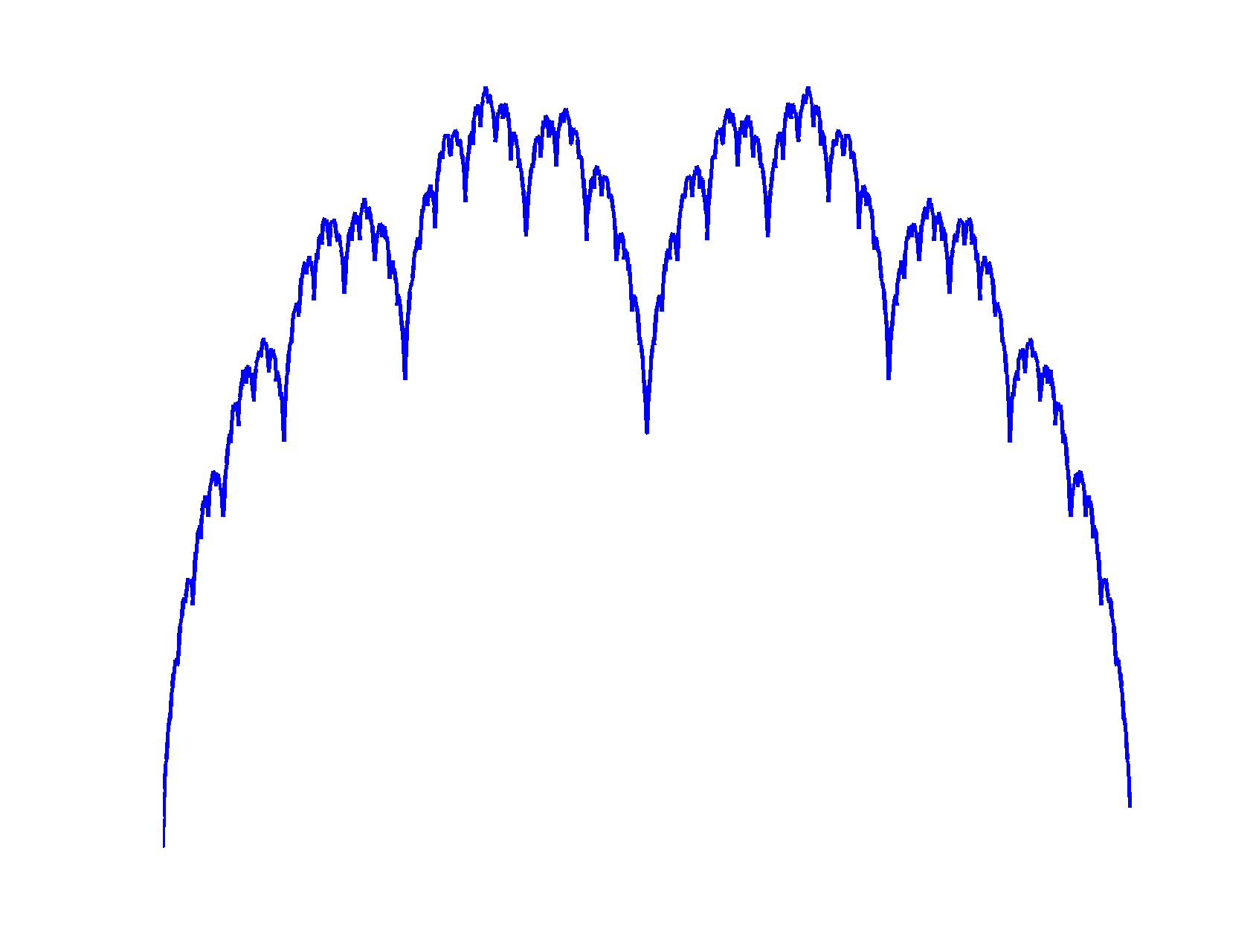}}
		\subfigure[$a=1/4$ (parabola)]{\includegraphics[width=0.45\linewidth]{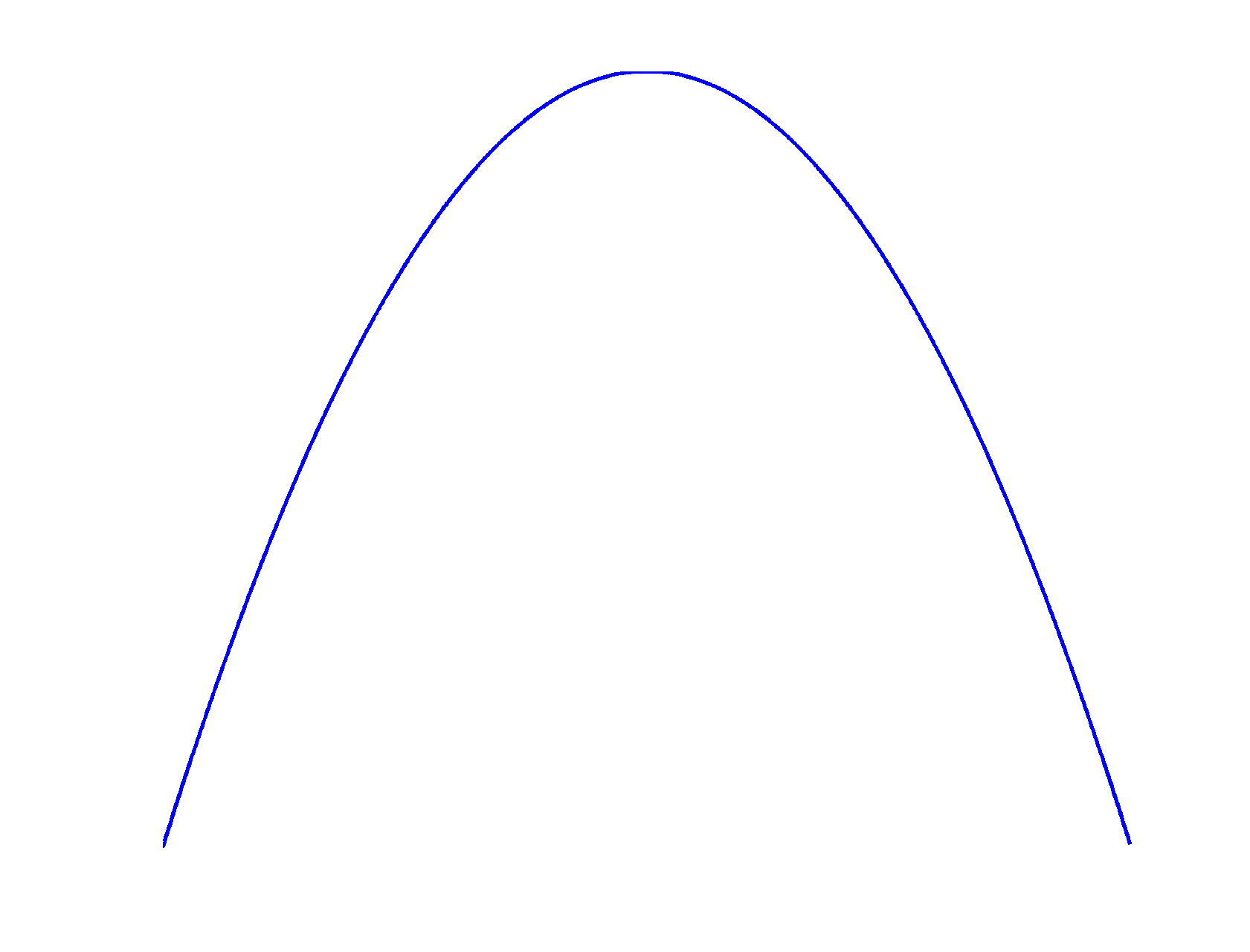}}
		\caption{Takagi-Landsberg curves for different values of the parameter $a$.}
		\label{fig:Odometer}
	\end{figure}
	
	From definition \eqref{eq:TakagiLandsberg}, it follows that the function $\mathcal{T}_a$ on the interval $[0,1]$ satisfies the following de Rham functional equations:
	\begin{equation}
		\begin{cases}
			\mathcal{T}_a(x/2) =  a \mathcal{T}_a(x) + x/2, \\
			\mathcal{T}_a(\frac{x+1}{2}) =  a \mathcal{T}_a(x) + \frac{1-x}{2}.
		\end{cases}
		\label{eq:deRhamForTa}
	\end{equation}
	% For dyadic rational \( x \) (i.e. \( x \) of the form \( x = \frac{n}{2^k} \) with \( n\in \mathbb{Z}_+ \) and \( k \in \mathbb{N} \)) we choose the representation ending in all zeros\footnote{Under this convention, for dyadic rational points, series \eqref{eq:TakagiLandsberg} is finite irrespective of $a$.}
	Such a system of equations uniquely determines the function on dyadic rationals (i.e., \( x \) of the form \( x = \frac{n}{2^k} \) with \( n\in \mathbb{Z}_+ \) and \( k \in \mathbb{N} \); for dyadic rational points we canonically choose the dyadic representation ending in all zeros\footnote{Under this convention, at any dyadic rational $x$, series \eqref{eq:TakagiLandsberg} is finite irrespective of $a$.}). The function can be continuously extended to the whole interval $[0,1]$ if a contraction argument applies.  More precisely, using Banach's fixed-point theorem, it was shown in \cite{Barnsley} (see also \cite{Girgensohn1994} and \cite{Girgensohn2012}) that any system of functional equations of the form
	\begin{equation}
		\begin{cases}
			f(x/2) =  a_0 f(x) + g_0(x), \\
			f(\frac{x+1}{2}) =  a_1 f(x) + g_1(x),
		\end{cases}
		\label{eq:coherence}
	\end{equation}
	provided that $\max\{|a_0|, |a_1|\} < 1$ and the consistency condition\footnote{This condition ensures that the equations remain consistent at  $x=1/2.$}
	\begin{equation}
		\label{eq:consistency_cond}
		a_0\frac{g_1(1)}{1-a_1} + g_0(1) = a_1\frac{g_0(0)}{1-a_0} + g_1(0),
	\end{equation}
	uniquely defines a continuous function on the interval $[0,1]$.
	
	% For dyadic rational \( x \) (i.e. \( x \) of the form \( x = \frac{k}{2^m} \) with \( k \in \mathbb{Z}_+ \) and \( m \in \mathbb{N} \)) we choose the representation ending in all zeros. When necessary, to avoid confusion, we write \( \varepsilon_n(x) \) instead of \( \varepsilon_n \).

	\section{Main Result}
	\label{Sec:MainResult}
	
	Let \( s_q(n), q \in \mathbb{R}, \) be the \( q \)-weighted sum of digits in the binary representation of a positive integer \( n = \sum_{i \geq 0} \omega_i 2^{i} \), equal to \( \sum_{i \geq 0} \omega_i q^{i+1} \), and let \( S_q(n) \) denote \( \sum_{j=1}^{n-1} s_q(j) \). For \( q=1 \), we have \( S_1(n) = S(n) \), and  the classical Trollope-Delange formula~\eqref{eq:TrollopeDelange} holds. Our result generalizes this to the \( q \)-weighted sums:
	
	\begin{theorem}
		\label{Th:mainTheorem}
		Let \( |q| > 1/2 \) and \( q \neq 1 \). Set \( a = 1/(2q) \). For any \( n \in \mathbb{N} \), the following expression holds:
		\begin{equation}
			\label{eq:generalizedTrollopeDelange}
			\frac{1}{n} S_q(n) = \frac{q}{2} \left( \frac{1-q^{[\log_2n]+1}}{1-q} - q^{[\log_2n]} \hat{F}_q(\log_2n) \right),
		\end{equation}
		where the 1-periodic function \( \hat{F}_q \) is given by
		\[
		\hat{F}_q(u) =  2^{1-u} \mathcal{T}_a(2^{-(1-u)}), \quad u \in [0,1].
		\]
	\end{theorem}
	
	To prove \eqref{eq:generalizedTrollopeDelange}, several approaches can be used. Our approach follows Girgensohn~\cite{Girgensohn2012}. The idea is to find the functional equations satisfied by the sequence \( S_q(n) \), which leads to the function \( \hat{F}_q \).
	
	\begin{proof}
		For \( k \in \mathbb{N} \), set \( p = 2^{k-1} \). Note that
		\begin{align}
			\label{eq:s(2j)}
			&s_q(2j) = q \, s_q(j),\\
			\label{eq:s(2j+1)}
			&s_q(2j+1) = q \, s_q(j) + q,\\
			\label{eq:s(j+p)}
			&s_q(j+p) = s_q(j) + q^k, \quad j = 0, 1, \dots, p-1,\\
			&s_q(j+p) = s_q(j) - q^k(1-q), \quad j = p, p+1, \dots, 2p-1.
		\end{align}
		
		Let \( [x] \) denote the integer part of \( x \), and \( \{x\} \) its fractional part. Set \( k_n = [\log_2n] \, \) and \( u_n = \{\log_2n\} \). Following the approach in \cite{Girgensohn2012}, denote\footnote{We try to stick to the short notation $p_n$ whenever it is sensible, i.e., unless we do not need to consider expressions like $p(n+p(n))$, etc.} by \( p_n = p(n) = 2^{k_n} \) the largest power to which 2 must be raised to get a number not exceeding \( n \), and by \( r_n \) denote \( q^{\log_2(p_n)} \equiv q^{k_n} \).
		
		It is easy to see that
		\begin{align}
			&p(2n) = 2p(n),\\
			&p(n+p(n)) = 2p(n),\\
			&p(n+2p(n)) = 2p(n).
		\end{align}
		
		One can verify that for any \( n \in \mathbb{N} \), the following relations hold\footnote{For illustration, here we obtain \eqref{eq:Sn+2p}. Expressions \eqref{eq:Sn2+p}-\eqref{eq:S2n} can be obtained analogously. After subtracting and adding the same quantity $S_q(2p_n)$ to $S_q(n+2p_n)$, we get
			\[S_q(n+2p_n)= s_q(2p_n+1) + s_q(2p_n+2)+\dots+s_q(2p_n+n)+ S_q(2p_n).\]
			Using \eqref{eq:s(j+p)} with $p=2 p_n = 2^{k_n+1}$  the latter writes as
			$s_q(1) + s_q(2)+\dots+s_q(n)+nq^{k_n+2}+S_q(2p_n)$.}:
		\begin{align}
			&S_q(n+2p_n) = S_q(n) + S_q(2p_n) + nq^{k_n+2}, \label{eq:Sn+2p}\\
			&S_q(n+p_n) = S_q(n) + (2q-1) S_q(p_n) - (n-p_n) q^{k_n+1} (1-q) + q p_n, \label{eq:Sn2+p}\\
			&S_q(2n) = 2q S_q(n) + nq. \label{eq:S2n}
		\end{align}

		The last of these relations implies that at points of the form \( p = 2^k \), the value of the function \( S_q \) is
		\begin{equation}
			\label{eq:SqAtP}
			S_q(p) = q \frac{1-q^k}{1-q} 2^{k-1} = q \frac{1-q^k}{1-q} \frac{p}{2}.
		\end{equation}
		
		Define the function \( G_q: \mathbb{N} \rightarrow \mathbb{R} \) by
		\begin{equation}
			G_q(n) = \frac{1}{p_n r_n} \left( S_q(n) - \frac{n}{p_n} S_q(p_n) \right).
			\label{eq:G_q_definition}
		\end{equation}
		
		\begin{figure}[t]
			\centering
			\includegraphics[scale=0.45]{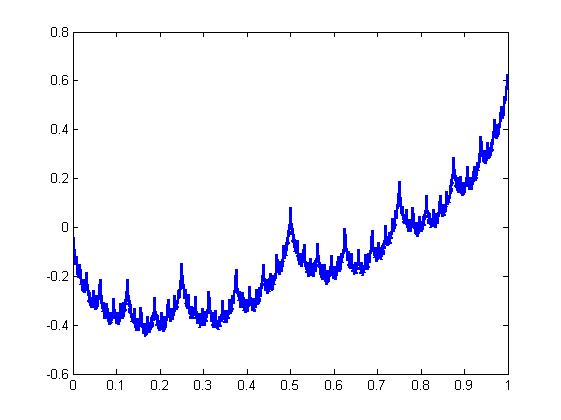}
			\caption{Graph of the function \( F_{2/3} \).}
			\label{fig:Odometer}
		\end{figure}
		
		The function \( G_q(n) \) satisfies the following relations:
		\begin{align}
			&G_q(2n) = G_q(n), \label{eq:G(2n)}\\
			&G_q(n+p_n) = \frac{1}{2q} G_q(n) + \frac{p_n-n}{4p_n} (3-2q), \label{eq:G(n+pn)}\\
			&G_q(n+2p_n) = \frac{1}{2q} G_q(n) + \frac{n}{4p_n} (2q-1). \label{eq:G(n+2pn)}
		\end{align}
		
		Relation \eqref{eq:G(2n)} follows directly from \eqref{eq:S2n}. We derive relation \eqref{eq:G(n+pn)}. For brevity, here we will write \( p \) for \( p_n = p(n) \), $p(n+p)$ instead of $p(n+p(n))$, \( r \) for \( r_n \), and $k$ for~$k_n$.
		\[
		\begin{split}
			&G_q(n+p) = \frac{1}{2qpr} \left( S_q(n+p) - \frac{n+p}{p(n+p)} S_q(p(n+p)) \right) \\
			&= \frac{1}{2qpr} \left( S_q(n) + (2q-1) S_q(p) - (n-p) q^{k+1} (1-q) + pq - \frac{n+p}{2p} (2q S_q(p) + pq) \right) \\
			&= \frac{1}{2q} G_q(n) + \frac{1}{2qpr} \left( \frac{p-n}{2} q + (1-q) \left( \left( \frac{n}{p} - 1 \right) S_q(p) - (n-p) q^{k+1} \right) \right) \\
			&= \frac{1}{2q} G_q(n) + (3-2q) \frac{p-n}{4p}.
		\end{split}
		\]
		Relation \eqref{eq:G(n+2pn)} can be derived as follows:
		\[
		\begin{split}
			&G_q(n+2p) = \frac{1}{2qpr} \left( S_q(n+2p) - \frac{n+2p}{2p} S_q(p(n+p)) \right) \\
			&= \frac{1}{2qpr} \left( S_q(n) - \frac{n}{p} S_q(p) - \frac{n}{2} q + q^{k+2} \right) \\
			&= \frac{1}{2q} G_q(n) + \frac{n}{2qpr} \left( (1-q) \frac{S_q(p)}{p} - \frac{q}{2} + q^{k+2} \right) \\
			&= \frac{1}{2q} G_q(n) + (2q-1) \frac{n}{4p}.
		\end{split}
		\]
		
		Denote by \( x_n = x(n) = 2^{u_n} - 1 = \frac{n-p_n}{p_n} \in [0,1) \). The following lemma, proved by Girgensohn in \cite{Girgensohn2012}, holds:
		
		\begin{lemma}
			\label{lemma:Girgensohn}
			Let \( G: \mathbb{N} \rightarrow \mathbb{R} \) be a function. For \( n \in \mathbb{N} \), set \( x = \frac{n-p_n}{p_n} \in [0,1) \) and define the function \( F \) by
			\[ F(x) = F \left( \frac{n-p_n}{p_n} \right) := G(n). \]
			The function \( F \) is well-defined in dyadic-rational points from the interval \([0,1)\) if and only if \( G(2n) = G(n) \) for all \( n \in \mathbb{N}. \)
		\end{lemma}
		
		The identities
		\begin{align}
			&\frac{x(n)}{2} = \frac{n-p_n}{2p_n} = x(n+p_n), \\
			&\frac{x(n)+1}{2} = \frac{n}{2p_n} = x(n+2p_n)
		\end{align}
		also hold. By Lemma \ref{lemma:Girgensohn}, the function \( F_q \), given by \( F_q(x_n) = G_q(n) \), is well-defined in dyadic-rational points from \([0,1)\). Using that $a = 1/(2q)$ identities \eqref{eq:G(n+pn)}--\eqref{eq:G(n+2pn)} can be rewritten as follows for \( x = x_n \):
		\begin{equation}
			\begin{cases}
				F_q(x/2) = a F_q(x) + (2q-3) \frac{x}{4}, \\
				F_q(\frac{x+1}{2}) = a F_q(x) + (2q-1) \frac{x+1}{4}.
			\end{cases}
			\label{eq:FbydeRham}
		\end{equation}
		% Recall that $|a| = |1/(2q)|<1$. From the system \eqref{eq:FbydeRham} we find that
		System~\eqref{eq:FbydeRham} satisfies consistency condition~\eqref{eq:consistency_cond} and defines \( F_q \) on dyadic rationals. The contraction condition \( |a| = |1/(2q)| < 1 \) allows us to extend it to the whole interval \( [0,1] \).  
		The function \( qx - \frac{1}{2} \mathcal{T}_a(x) \) also satisfies de Rham equations~\eqref{eq:FbydeRham} and thus coincides with \( F_q \):  
		\begin{equation}  
			\label{eq:solutionOfdeRhamSystem}  
			F_q(x) = qx - \frac{1}{2} \mathcal{T}_a(x) = q \left( \frac{1+x}{2} - \mathcal{T}_a \left( \frac{x+1}{2} \right) \right).  
		\end{equation}

		Using the identity \( S_q(n) = r_n p_n G_q(n) + \frac{n}{p_n} S_q(p_n) \), we arrive at
		\begin{equation}
			\label{eq:generalizedTrollopeDelange2a}
			\frac{1}{n} S_q(n) = q^{k_n} \frac{F_q(x_n)}{x_n+1} + \frac{q}{2} \frac{1-q^{k_n}}{1-q}.
		\end{equation}
		For brevity we will write  \( k \) for \( k_n  \).
		Let \( x = 2^u - 1 \) and, accordingly, \( \frac{1}{x+1} = 2^{-u} \). We can write \( \frac{F_q(x)}{x+1} = q \left( \frac{1}{2} - 2^{-u} \mathcal{T}_a(2^{u-1}) \right) \) and represent \eqref{eq:generalizedTrollopeDelange2a} as
		\[
		\frac{1}{n} S_q(n) = \frac{q}{2} \frac{1-q^k}{1-q} + q^{k+1} \left( \frac{1}{2} - 2^{-u} \mathcal{T}_a(2^{u-1}) \right) = \frac{q}{2} \left( \frac{1-q^k}{1-q} + q^k \left( 1 - 2^{1-u} \mathcal{T}_a(2^{u-1}) \right) \right),
		\]
		which leads to the desired formula:
		\begin{equation*}
			\label{eq:generalizedTrollopeDelange2b}
			\frac{1}{n} S_q(n) = \frac{q}{2} \left( \frac{1-q^{[\log_2n]+1}}{1-q} - q^{[\log_2n]} 2^{1-u} \mathcal{T}_a(2^{u-1})  \right).
		\end{equation*}
	\end{proof}
	
	\emph{Remark} 1. System \eqref{eq:FbydeRham} lets one define  \( F_q \) in the dyadic rational points regardless of the value of $a$. This results\footnote{ It is  straightforward to check that $\mathcal{T}_a(n/(2p_n)) = \frac{1}{a^{k_n+1}} \sum_{i=1}^{k_n+1} a^i \tau(n/2^i). $ Similarly, the expression on the right-hand side makes sense for any $a$. Then, we can use \eqref{eq:generalizedTrollopeDelange2a} to obtain \eqref{eq:1/nS_at_Dyadic_points}.}

	\begin{equation}
		\frac{1}{n} S_q(n) = \frac{q}{2} \frac{1-q^{k_n+1}}{1-q} -     \frac{1}{2n} \sum_{i=1}^{k_n+1} (2q)^i \tau(n/2^i). 
		\label{eq:1/nS_at_Dyadic_points}
	\end{equation}
	However, the condition \( |q| > 1/2 \) in Theorem \ref{Th:mainTheorem}  cannot be omitted:  Banach's fixed-point theorem, which relied on the contraction principle, cannot be applied to  \eqref{eq:FbydeRham} without this condition. It is then an expanding mapping instead, which results in the absence of continuous solutions on the interval $[0,1]$.

	\emph{Remark} 2. The special case of formula \eqref{eq:generalizedTrollopeDelange} for \( q = -1 \) was obtained by Kr\"uppel in 2008, see Theorem 5.1 in \cite{Kruppel2008}. In it, the graph of the function \( \mathcal{T}_{-1/2} \) was called the alternating sign Takagi curve. 
	
	\emph{Remark} 3.  The case of \( q = 1/2 \) was studied in connection with the discrepancy of the van der Corput sequence. In this case, it was shown, see \cite{Larcher-et-al-2005}, Theorem~$4,$  that
	\[
	\frac{1}{n} S_{1/2}(n) = \frac{1 - D^*_n}{2} = \frac{1}{2} \left( 1 - \frac{1}{n} \left( 1 + \sum_{j=1}^{k_n} \tau \left( \frac{n}{2^j} \right) \right) \right),
	\]
	where \( D^*_n \) denotes the star discrepancy of the van der Corput sequence. The expression on the right-hand side can be obtained as a special case of \eqref{eq:1/nS_at_Dyadic_points}.

	\begin{corollary*}
		% \label{Th:mainTheorem}
		Let \( q > 1/2 \) and \( q \neq 1 \) and $a = 1/(2q)$. The following generalized Trollope-Delange formula holds:
		\begin{equation*}
			\label{eq:generalizedTrollopeDelange2}
			\frac{1}{n} S_q(n) = \frac{q}{2} \left( \frac{1-q^{\log_2n}}{1-q} + q^{\log_2n} \tilde{F}_q(\log_2n) \right),
		\end{equation*}
		where the 1-periodic function \( \tilde{F}_q \) is given by
		\begin{equation}
			\label{eq:F_tilde}
			\tilde{F}_q(u) = \frac{1-q^{1-u}}{1-q} - q^{-u} 2^{1-u} \mathcal{T}_a(2^{-(1-u)}), \quad u \in [0,1].
		\end{equation}
	\end{corollary*}
	\begin{proof}
		
		We use the representation derived in Theorem~1,
		\[
		\frac{1}{n} S_q(n) = \frac{q}{2} \left( \frac{1-q^k}{1-q} + q^k \left( 1 - 2^{1-u} \mathcal{T}_a(2^{u-1}) \right) \right),
		\]
		where, as above, $u = \{\log_2n\}$ and  the fact that \( q^k = q^{\log_2n} q^{-u} \),  to obtain the desired formula:
		\begin{equation*}
			\label{eq:generalizedTrollopeDelange2b}
			\frac{1}{n} S_q(n) = \frac{q}{2} \left( \frac{1-q^{\log_2n}}{1-q} + q^{\log_2n} \left( \frac{1-q^{1-u}}{1-q} - 2^{1-u} q^{-u} \mathcal{T}_a(2^{u-1}) \right) \right).
		\end{equation*}
	\end{proof}
	
	Figure \ref{fig:F_tilde} further illustrates graphs of the "generalized Trollope-Delange" function $\tilde{F}_q$ for different values of the weight parameter $q;$ in case of $q=1$ we use expression~\eqref{eq:TD_function}, which can also be obtained by going to the limit in \eqref{eq:F_tilde}. It follows from Section \ref{Sec:TakagiLandsberg} that for $q>1$ ($a<1/2$), functions $\tilde{F}_q$ are almost everywhere differentiable; and also $\tilde{F}_2\equiv 0.$
	\begin{figure}[h!]
		\centering
		\includegraphics[scale=0.38]{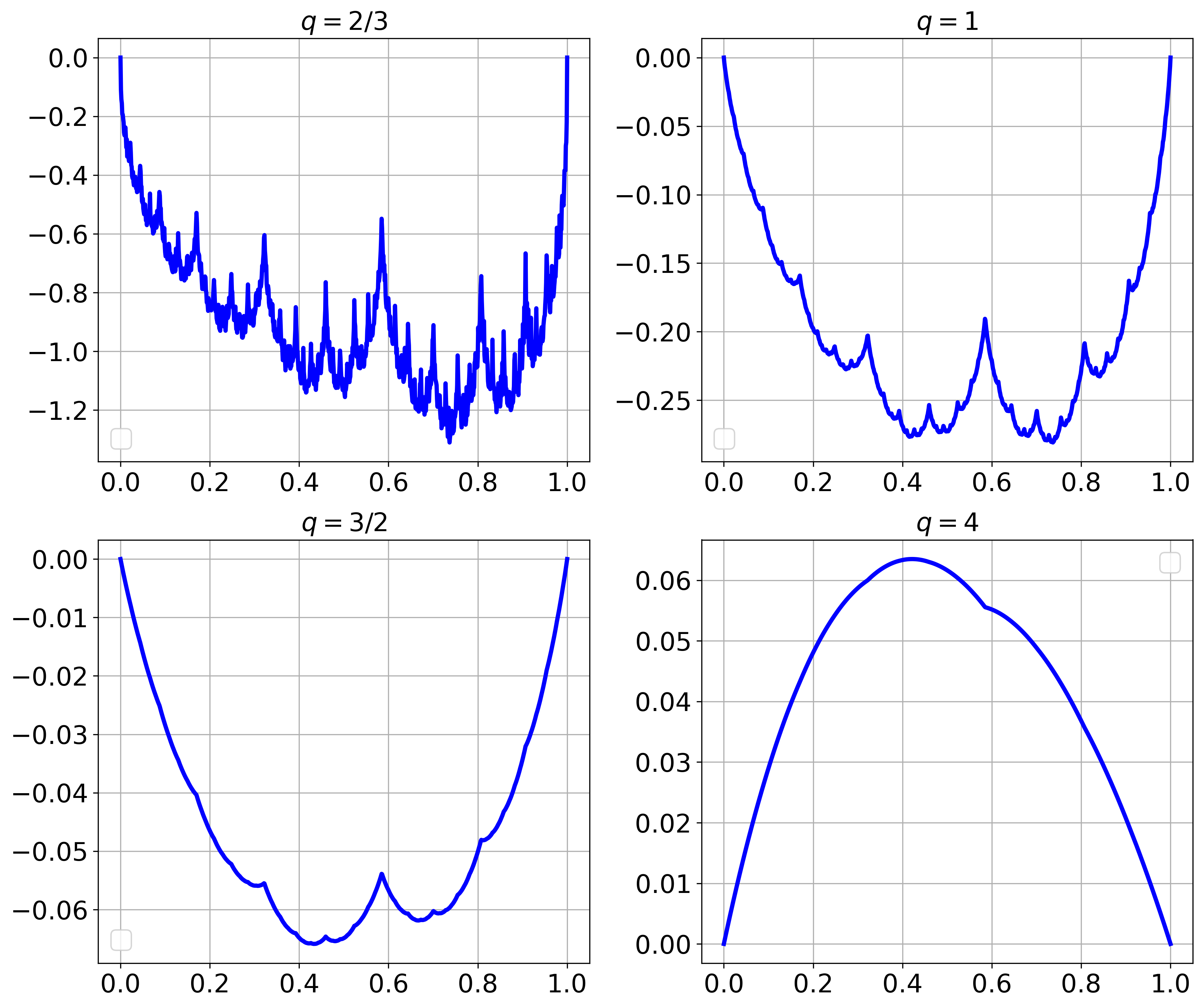}
		\caption{Graphs of the function \( \tilde{F}_{q},\, q \in  \{2/3, 1, 3/2, 4\} \).}
		\label{fig:F_tilde}
	\end{figure}

	\section{Extension to Complex Weights}
	\label{Sec:complex_q}
	The proofs of Theorem~\ref{Th:mainTheorem} extend almost verbatim to the complex setting. Specifically, if 
	\[
	q\in\mathbb{C}\quad \text{with}\quad |q|>1/2,
	\]
	so that, with 
	\[
	a=\frac{1}{2q},
	\]
	we have \(|a|<1\), then the \(q\)-weighted sum-of-digits function
	\[
	s_q(n)=\sum_{i\ge0}\omega_i\,q^{\,i+1},
	\]
	and its cumulative sum
	\[
	S_q(n)=\sum_{k=0}^{n-1}s_q(k),
	\]
	are well defined (with complex values), and the generalized Trollope-Delange formula
	\[
	\frac{1}{n}S_q(n)=\frac{q}{2}\left(\frac{1-q^{\lfloor\log_2n\rfloor+1}}{1-q}-q^{\lfloor\log_2n\rfloor}\,\hat{F}_q(\log_2n)\right),
	\]
	with
	\[
	\hat{F}_q(u)=2^{1-u}\,\mathcal{T}_a\Bigl(2^{-(1-u)}\Bigr),\quad a=\frac{1}{2q},
	\]
	remains valid. Since the sequence \( (|2a|^n)_{n\geq 0} \) belongs to \(\ell^2\) when \(|a| < 1/2\), it follows from \cite{Kono1987} that for such values of parameter $a$ the function \( \mathcal{T}_a \) is almost everywhere differentiable.

	Figure~\ref{fig:ComplexTA} illustrates the nontrivial behavior in this extended setting.  We plot the complex-valued Takagi-Landsberg function $\mathcal{T}_a(x)$ for several choices of $q$:
	\[
	q=i,\quad q=\frac{1}{2}+\frac{i}{2},\quad \text{and}\quad q=\frac{1}{2}-\frac{i}{2}.
	\]
	
	% In each case, \(\mathcal{T}_a(x)\) (with \(a=1/(2q)\)) is complex-valued, and its real and imaginary parts exhibit intricate oscillations and fractal structure (see Figure~\ref{fig:ComplexTA}). 
	
	\begin{figure}[t]
		\centering
		\includegraphics[width=1.0\linewidth]{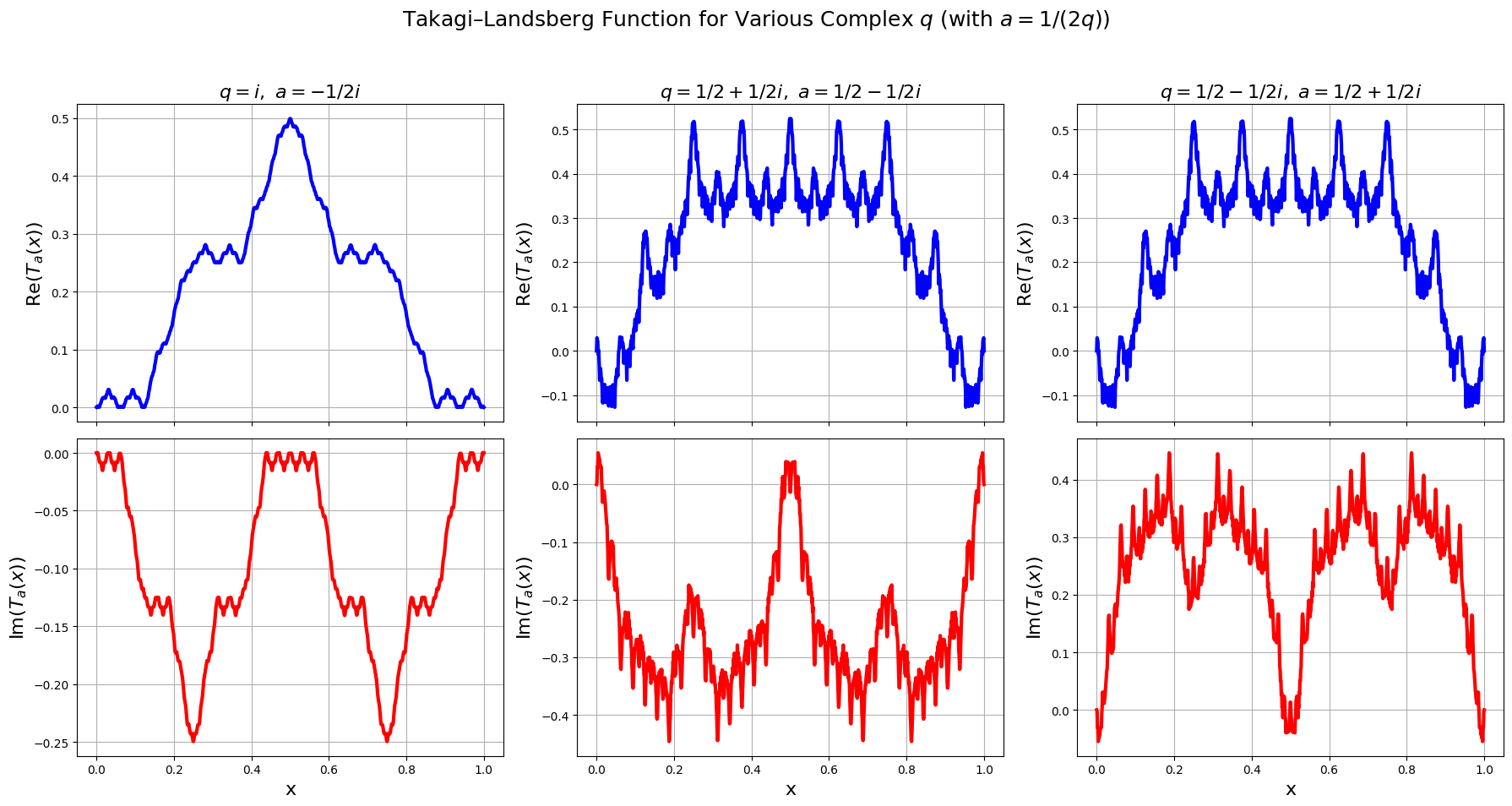}
		\caption{Takagi--Landsberg curves for complex values of \(a\).}
		\label{fig:ComplexTA}
	\end{figure}

	\section{The Trollope-Delange Formula and the Limiting Curve of the Sequence $s_q(n)$}
	\label{Sec:LimitingCurves}
	Let $(a_n)_{n \geq 0}$ be a real sequence. Extend the partial sums $S(n) = \sum_{j=0}^{n-1} a_j$ to real-valued arguments using linear interpolation and consider the sequence of continuous functions $\varphi_n: [0,1] \rightarrow [-1,1]$ defined by the equation $\varphi_n(t) = \frac{S(t \cdot n) - t \cdot S(n)}{R_n}$, where the normalization factor $R_n$ is chosen to be $\max_{t \in [0,1]} |S(t \cdot n) - t \cdot S(n)|$, if this maximum is not zero, otherwise, $R_n$ is set to one. Note that due to the subtracted term $t\cdot S(n)$, the functions $\varphi_n$ take zero values at the endpoints of the interval $[0,1]$.
	
	If there exists a sequence of positive integers $(l_n)$ such that the sequence of functions $\varphi_{l_n}$ converges to a (continuous) function $\varphi$ in the $\sup$-metric on $[0,1]$, then the graph of the limiting function $\varphi$ is called a \emph{limiting curve of the sequence} $a_n$, the sequence $l_n$ is called a \emph{stabilizing sequence}, and the sequence\footnote{What matters here is the growth rate of \( R_{l_n} \).} $R_{l_n}$ is called a \emph{normalizing sequence}. The authors of \cite{DeLaRue}  proposed this definition to describe fluctuations of ergodic sums of certain functions while studying dynamical systems, particularly the Pascal adic transformation. This definition was also considered in the works of the author \cite{Min17}, \cite{Min19}; see also the references therein. The idea is to consider the limit of the normalized difference of the partial sum and the linear function along the $X$ and $Y$ axes.
	
	Consider the initial sequence $a_n$ to be the sequence $s_q(n)$. 
	Using the Trollope-Delange formula, we can demonstrate the following proposition, which has also been proven by an alternative method in \cite{Min19}.
	
	\begin{proposition}
		\label{Prop:MainProp2}
		Let \( q \in \mathbb{C} \) with \( |q| > \frac{1}{2} \) (in particular, this includes the real case), and set  $a = 1/(2q)$. Choose any $N \in \mathbb{N}$ and fix $l = 2^N$, and $R = (2q)^{N-1}$. Consider the dyadic-rational points \( t_j = \frac{j}{2^{N-1}} \) for \( j = 0, 1, 2, \dots, 2^{N-1} \), and define 
		\[
		\varphi_l(t_j) = \frac{S_q(t_j \cdot n) - t_j \cdot S_q(n)}{R}.
		\]
		Then, the following equality holds
		\begin{equation}
			\varphi_l(t_j) = -q \mathcal{T}_a(t_j).
			\label{eq:varphi_l}
		\end{equation}
	\end{proposition}
	
	\begin{proof}
		
		\begin{figure}[h!]
			\centering
			\includegraphics[scale=0.95]{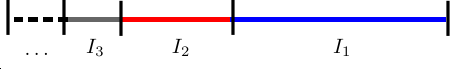}
			\caption{Partitioning of the interval $I = [1, l)$.}
			\label{fig:Intervals}
		\end{figure}
		
		Define the function $\tilde{G}_q(n)$ on $I$ by
		\[
		\tilde{G}_q(n) = \frac{1}{R} \left( S_q(n) - \frac{n}{2p} S_q(2p) \right) = \varphi_l \left( \frac{n}{l} \right).
		\]
		Note that $r,p$ (and thus $ rp = R$) are set here \textit{not} to depend on $n$, which contrasts to \eqref{eq:G_q_definition}. 
		
		Let $l_m = l / 2^m$. Partition the interval $I = [1, l)$ into $N$ non-overlapping subintervals $(I_m)_{m=1}^N$, $I_m = [l_m, l_{m-1})$, so that $I = \cup_{m=1}^N I_m$, see Figure~\ref{fig:Intervals}. To prove the relation \eqref{eq:varphi_l} on the entire $I$, we sequentially iterate through each of the intervals $I_m, m = 1, 2, \dots, N$ and demonstrate the validity of the relation \eqref{eq:varphi_l}. Below, we verify the expression \eqref{eq:varphi_l} first on $I_1$, and then on an arbitrary $I_m, m \geq 1$.
		
		On the interval $I_1$ (i.e., for $\frac{l}{2} \leq n < l$), the quantities $p = p_n$, $r = r_n$, and $x = x_n$, defined in the proof of Theorem 1, can also be given by simpler relations: $p = l/2 = 2^{N-1}$, $r = q^{N-1}$, and $\frac{x+1}{2} = \frac{n}{l}$. 
		Using the relations \eqref{eq:Sn+2p}--\eqref{eq:Sn2+p}, \eqref{eq:SqAtP}, and \eqref{eq:solutionOfdeRhamSystem}, we rewrite the expression for $\tilde{G}_q(n), n \in I_1,$ as follows:
		
		% \[
		% \tilde{G}_q(n) = \frac{1}{pr} \left( S_q(n) - \frac{nq}{2p} (2S_q(p) + p) \right) = \frac{1}{pr} \left( S_q(n) - \frac{n}{p} S_q(p) + (1-q) \frac{n}{p} S_q(p) - q \frac{n}{2} \right)
		% \]
		% \[
		% = G_q(n) - \frac{nq}{2pr} + (1-q) \frac{n}{p} \frac{S_q(p)}{pr} = G_q(n) - \frac{qn}{2pr} + \frac{nq (1-q^{N-1})}{2pr}
		% \]
		% \[
		% = G_q(n) - \frac{qn}{2p} = F_q(x) - q \frac{x+1}{2} = -q \mathcal{T}_a \left( \frac{n}{l} \right).
		% \]
		
		\begin{align*}
			\tilde{G}_q(n) &= \frac{1}{pr} \left( S_q(n) - \frac{nq}{2p} (2S_q(p) + p) \right) \\
			&= \frac{1}{pr} \left( S_q(n) - \frac{n}{p} S_q(p) + (1-q) \frac{n}{p} S_q(p) - q \frac{n}{2} \right) \\
			&= G_q(n) - \frac{nq}{2pr} + (1-q) \frac{n}{p} \frac{S_q(p)}{pr} \\
			&= G_q(n) - \frac{qn}{2pr} + \frac{nq (1-q^{N-1})}{2pr} \\
			&= G_q(n) - \frac{qn}{2p} = F_q(x) - q \frac{x+1}{2} =-q \mathcal{T}_a \left( \frac{n}{l} \right).
		\end{align*}
		
		Proceed to the interval $I_m$, defined by the inequalities $l_m \leq n < l_{m-1}$, where $m = 1, 2, \dots, N$. Set $N_m = N-m,$ $r_m = q^{m-1},$ $t = t_n = \frac{n}{l}$. From the relations \eqref{eq:S2n} and \eqref{eq:deRhamForTa}, it is easy to obtain that
		\begin{equation*}
			S_q(l) = (2q)^m S_q(l_m) + 2^{m-1} l_m q (1+q+\dots+q^{m-1})
			\label{eq:Sl}
		\end{equation*}
		and
		\begin{equation*}
			\mathcal{T}_a(t) = (2q)^m \mathcal{T}_a \left( \frac{t}{2^m} \right) - t q (1+q+\dots+q^{m-1}).
			\label{eq:Tl}
		\end{equation*}
		
		Then, adding and subtracting $\frac{n}{l_m} S_q(l_m)$ we can show that on the interval $I_m$ it holds
		% \[
		% \tilde{G}_q(n) \equiv \frac{1}{(2q)^{N-1}} \left( S_q(n) - \frac{n}{l} S_q(l) \right) = \frac{1}{(2q)^{N-1}} \left( S_q(n) + \frac{n}{l_m} S_q(l_m) - \frac{n}{l_m} S_q(l_m) - \frac{n}{2^m l_m} S_q(l) \right)  
		% \]
		% \[
		% = \frac{1}{(2q)^{N-1}} \left( (2q)^{N-m} G_q(n) + \frac{n}{l_m} S_q(l_{m-1}) - \frac{n}{2^m l_m} \left( (2q)^m S_q(l_m) + 2^{m-1} l_m (q+\dots+q^m) \right) \right) 
		% \]
		% \[
		% =\frac{1}{(2q)^{m-1}} G_q(n) + \frac{n}{l_m} \frac{1}{(2q)^{N-1}} \left( (1-q^m) S_q(l_m) - \frac{2^{m-1} l_m}{2^m} (q+\dots+q^m) \right).
		% \]
		% Using $S_q(l_m) = q \frac{1-q^{N-m}}{1-q} \frac{l_m}{2}$, we can continue the chain of equalities and write that
		% \[
		% \tilde{G}_q(n) = \frac{1}{(2q)^{m-1}} G_q(n) - \frac{q^{N-m}}{(2q)^{N-1}} \frac{n}{l_m} \frac{l_m}{2} (q+\dots+q^m).
		% \]
		
		\begin{align*}
			\tilde{G}_q(n) &\equiv \frac{1}{(2q)^{N-1}} \left( S_q(n) - \frac{n}{l} S_q(l) \right) \\
			&= \frac{1}{(2q)^{N-1}} \left( S_q(n) + \frac{n}{l_m} S_q(l_m) - \frac{n}{l_m} S_q(l_m) - \frac{n}{2^m l_m} S_q(l) \right)  \\
			&= \frac{1}{(2q)^{N-1}} \bigg( (2q)^{N-m} G_q(n) + \frac{n}{l_m} S_q(l_{m-1}) \\
			&\quad - \frac{n}{2^m l_m} \big( (2q)^m S_q(l_m) + 2^{m-1} l_m (q+\dots+q^m) \big) \bigg) \\
			&= \frac{1}{(2q)^{m-1}} G_q(n) + \frac{n}{l_m} \frac{1}{(2q)^{N-1}} 
			\left( (1-q^m) S_q(l_m) - \frac{2^{m-1} l_m}{2^m} (q+\dots+q^m) \right).
		\end{align*}
		
		Using relation \eqref{eq:solutionOfdeRhamSystem}, which can now be written as $G_q(n) = F(x_n) = F(2^{m-1} t) = q \left( 2^{m-1} t - \mathcal{T}_a(2^{m-1} t) \right),$ where $\frac{n}{l_{m-1}} = 2^{m-1} t.$ We then get
		$\tilde{G}_q(n) = -\frac{1}{(2q)^{m-1}} F(x_n) - \frac{n}{l} \frac{1}{q^{m-1}} (q+\dots+q^m) = \frac{q}{(2q)^{m-1}} \Big( 2^{m-1} t - \mathcal{T}_a(2^{m-1} t) - 2^{m-1} t (1+q+\dots+q^{m-1}) \Big).$
		This can be rewritten, using expression \eqref{eq:Tl}, as follows:
		$\tilde{G}_q(n) = \frac{q}{(2q)^{m-1}} ( 2^{m-1} t - (2q)^{m-1} \mathcal{T}_a(t) + 2^{m-1} t (q+\dots+q^{m-1}) $ $- 2^{m-1} t (1+q+\dots+q^{m-1}) ) = -q \mathcal{T}_a(t) = -q \mathcal{T}_a \left( \frac{n}{l} \right).$
		
	\end{proof}
	
	Define the sequence $l_n$ as $2^n$. Since by Proposition \ref{Prop:MainProp2}, the (continuous) function $\varphi_{l_n}$ coincides in all dyadic points of the form $t_j = \frac{j}{2^{n-1}}, j \in \{0, 1, 2, \dots, 2^{n-1}\}$ of the interval $[0,1]$ with the continuous function $-q \mathcal{T}_a$, the same holds for the uniform limit $\lim_{n} \varphi_{l_n}$. Therefore, the function $-q \mathcal{T}_a, a = 1/(2q),$ is the limiting function of the sequence $s_q(n)$ with stabilizing and normalizing sequences given by $l_n = 2^n, R_n = (2q)^{n-1}, n \in \mathbb{N}.$
	
	\paragraph{Limiting curves for the dyadic odometer.} 
	In 1981 work \cite{Vershik1981}, Anatoly Vershik introduced into the literature the notion of adic transformations.  
	The simplest example---the dyadic odometer---can be defined as follows:  
	Let \( \mathbb{Z}_2 = \prod_{0}^{\infty} \{0,1\} \) be the compact additive group of dyadic integers with the Haar measure \( \mu \).  
	The odometer transformation \( T \) is the translation \( T \omega = \omega + 1 \). The dyadic odometer is a well-studied dynamical system known to be uniquely ergodic, of rank one, and with a discrete dyadic spectrum. We focus on the fluctuations in the ergodic theorem for the sum-of-digits function. 
	
	Let \( q \) be a parameter (real or complex) such that 
	\[
	\frac{1}{2} < |q| < 1.
	\]
	For \( \omega = (\omega_0, \omega_1, \omega_2, \dots) \in \mathbb{Z}_2 \), we denote by \( \mathbf{s}_q \) the \( q \)-weighted sum of the coordinates\footnote{Note that the unweighted sum-of-digits function \( \mathbf{s}_1 \) can only be defined on a subset of dyadic rationals, which has measure zero. This justifies the consideration of weighted sums.} defined by 
	\[
	\mathbf{s}_q(\omega) = \sum_{i \geq 0} \omega_i\,q^{\,i+1}.
	\]
	For \( \omega \in \mathbb{Z}_2 \), we define the partial sums 
	\[
	S_{q,\omega}(n) = \sum_{j=0}^{n-1} \mathbf{s}_q(T^j \omega).
	\]
	The Birkhoff-Khinchin ergodic theorem states that for $ \mu $-almost every $ \omega $
	\[
	\lim\limits_{n\to \infty} \frac1n S_{q,\omega}(n) = \mathbb{E} \,\mathbf{s}_q.
	\]
	
	To analyze fluctuations around this mean behavior, we consider deviations of the form:
	\[
	\varphi_{q, \omega}^{l,R}(t) = \frac{S_{q,\omega}(t \cdot l) - t \cdot S_{q,\omega}(l)}{R}, \, t\in [0,1]. 
	\]
	Proposition \ref{Prop:MainProp2} can be used to prove (see \cite{Min19}) the following result:
	
	\begin{theorem}
		For \( \mu \)-almost every \( \omega \), there exists a stabilizing sequence \( l_n = l_n(\omega) \) and normalizing sequence \( R_n = R_n(\omega) \) such that \( \varphi_{q, \omega}^{l_n, R_n} \)  
		converges in the sup-metric on \( [0,1] \) to the function \( -T_a \) with \( a = \frac{1}{2q} \).
	\end{theorem}
	
	\section{Open questions}
	
	In this paper, we derived a generalized exact Trollope-Delange formula describing the behavior of the sequence $\frac{1}{n}S_q(n)$ for $|q| > 1/2$, which can be considered the "first moment" for the sequence $s_q(n)$. For $|q| \leq  1/2$  our approach gives an expression defined only on dyadic rational points. It is natural to ask whether alternative representations can be obtained in this case. Additionally, exploring higher moments like $\sum_{j=0}^{n-1} s_q^k(j)$ for $k \geq 2$, exponential sums $\sum_{j=0}^{n-1} \exp(t , s_q(j))$, and binomial sums $\sum_{j=0}^{n-1} \binom{s_q(j)}{m}$ are of significant interest. These questions are left for future research.
	
	\section*{Acknowledgment}
	
	I sincerely thank the anonymous referee for their valuable feedback, which has led to numerous improvements in the clarity and rigor of this work. An early version of the work was supported by the RFBR grant 17-01-00433.
	%Any remaining errors are my responsibility.
	
	% \section*{Appendix}
	% \subsection{Derivation of \eqref{eq:Sn+2p}-\eqref{eq:S2n}}
	% % \begin{align}
	% % &S_q(n+2p_n) = S_q(n) + S_q(2p_n) + nq^{k_n+2}, \label{eq:Sn+2p}\\
	% % &S_q(n+p_n) = S_q(n) + (2q-1) S_q(p_n) - (n-p_n) q^{k_n+1} (1-q) + q p_n, \label{eq:Sn2+p}\\
	% % &S_q(2n) = 2q S_q(n) + nq. \label{eq:S2n}
	% % \end{align}
	% Subtracting and adding the same quantity, we have
	% \[S_q(n+2p)=S_q(n+2p) - S_q(2p) + S_q(2p) = s_q(2p+1) + s_q(2p+2)+\dots+s_q(2p+n)+ S_q(2p)\]
	% Using \eqref{eq:s(j+p)} $n$ times the latter writes as
	% \[s_q(1) + s_q(2)+\dots+s_q(n)+nq^{k+1}+S_q(2p)\]


\begin{thebibliography}{10}
		
		
		\bibitem{Allaart11} P. Allaart, \emph{An inequality for sums of binary digits, with application to Takagi functions}, Journal of Mathematical Analysis and Applications, 381:2 (2011), 689--694.
		
		\bibitem{AllaartKawamura2012} P. Allaart, K. Kawamura, \emph{The Takagi function: a survey.} Real Anal. Exch., 37:1 (2012), 1--54.
		
		\bibitem{Barnsley} M. Barnsley, \emph{Fractal functions and interpolation}, Constr. Approx. 2 (1986), 303--329.
		
		\bibitem{Boros} Z. Boros, \emph{An inequality for the Takagi function}, Mathematical Inequalities and Applications 11:4 (2008), 757--765.
		
		\bibitem{Coquet} J. Coquet, \emph{Power sums of digital sums}, J. Number Theory 22 (1986), 161--176.
		
		\bibitem{Delange} H. Delange, \emph{Sur la fonction sommatoire de la fonction somme des chiffres}, Enseign. Math. 21:2 (1975), 31--47.
		
		\bibitem{Flajolet} P. Flajolet, P. Grabner, P. Kirschenhofer, H. Prodinger, R. F. Tichy, \emph{Mellin transforms and asymptotics: digital sums}, Theoret. Comput. Sci. 123 (1994), 291--314.
		
		\bibitem{Girgensohn1994} R. Girgensohn, \emph{Nowhere differentiable solutions of a system of functional equations}, Aequationes mathematicae, 47:1 (1994), 89--99.
		
		\bibitem{Girgensohn2012} R. Girgensohn, \emph{Digital Sums and Functional Equations}, Integers, 12:1, (2012), 141--160.
		
		\bibitem{HataYamaguti} M. Hata and M. Yamaguti, \emph{The Takagi function and its generalization}, Japan J. Appl. Math., 1 (1984), 183--199.
		
		% \bibitem{Larcher-et-al-2008} R. Hofer, G. Larcher, F. Pillichshammer, \emph{Average growth-behavior and distribution properties of generalized weighted digit-block-counting functions}, Monatshefte f{\"u}r Mathematik, 154:3 (2008), 199--230.
		
		
		\bibitem{Hofer-et-al-2008} R. Hofer, G. Larcher, F. Pillichshammer, \emph{Average growth-behavior and distribution properties of generalized weighted digit-block-counting functions}, Monatshefte f{\"u}r Mathematik, 154:3 (2008), 199--230.
		
		\bibitem{DeLaRue} \'E. Janvresse, T. de la Rue, Y. Velenik, \emph{Self-similar corrections to the ergodic theorem for the Pascal-adic transformation}, Stoch. Dyn., 5:1 (2005), 1--25.
		
		\bibitem{Kono1987} N. K{\^o}no, \emph{On generalized Takagi functions}, Acta Mathematica Hungarica, 49 (1987), 315--324.
		
		
		\bibitem{allaart2011inequality} P. C. Allaart, \emph{An inequality for sums of binary digits, with application to Takagi functions}, Journal of Mathematical Analysis and Applications, 381(2) (2011), 689--694.
		
		\bibitem{Kruppel2008} M. Kr\"uppel, \emph{Takagi's continuous nowhere differentiable function and binary digital sums}, Rostock. Math. Kolloq., 63 (2008), 37--54.
		
		\bibitem{Kruppel} M. Kr\"uppel, \emph{De Rham's singular function, its partial derivatives with respect to the parameter and binary digital sums}, Rostocker Math. Kolloq., 64 (2009), 57--74.
		
		\bibitem{Lagarias} J. C. Lagarias, \emph{The Takagi function and its properties}, in: Functions in Number Theory and Their Probabilistic Aspects, 153--189, (2012).
		
		\bibitem{Landsberg} G. Landsberg, \emph{On the differentiability of continuous functions (Uber die Differentiierbarkeit stetiger Funktionen)}, Jahresber. Deutschen Math. Verein. 17 (1908), 46--51.
		
		\bibitem{Larcher-et-al-2005} G. Larcher, F. Pillichshammer, \emph{Moments of the Weighted Sum-of-Digits Function}, Quaestiones Mathematicae, 28:3 (2005), 321--336.
		
		\bibitem{Min17} A. Minabutdinov, \emph{Limiting curves for polynomial adic systems}, Zap. nauchn. sem. POMI,  448 (2016), 177--200.
		
		
		\bibitem{Min19} A. Minabutdinov, \emph{Limiting curves for the dyadic odometer}, Journal of Math. Sciences, {247} (2020), 688--695.
		
		
		
		\bibitem{deRham1956} G. de Rham, \emph{Sur quelques courbes definies par des equations fonctionnelles}, Rend. Sem. Mat., 16 (1956), 101--113.
		
		\bibitem{TaborTabor2009} J. Tabor, J. Tabor, \emph{Takagi functions and approximate midconvexity}, Journal of Mathematical Analysis and Applications, 356:2 (2009), 729--737.
		
		\bibitem{Takagi1903} T. Takagi, \emph{A simple example of the continuous function without derivative}, Proc. Phys.-Math. Soc., 1 (1903), 176--177.
		
		\bibitem{Trollope} E. Trollope, \emph{An explicit expression for binary digital sums}, Math. Mag. 41 (1968), 21--25.
		
		\bibitem{Vershik1981} A. M. Vershik, \emph{Uniform algebraic approximations of shift and multiplication operators}, Sov. Math. Dokl. 24 (1981), 97--100.
		
		
		
	\end{thebibliography}
\end{document}